\documentclass[11pt]{article}

\usepackage{authblk,amsfonts,amsmath,amssymb,amsthm,dsfont,color,fullpage}
\usepackage[round]{natbib}
\newtheorem{theorem}{Theorem}
\newtheorem{lemma}{Lemma}
\newtheorem{proposition}{Proposition}
\theoremstyle{definition}

\theoremstyle{remark}
\newtheorem{remark}{Remark}

\DeclareMathOperator{\diag}{diag}

\DeclareMathOperator{\tr}{tr}

\def\E{\mathbb{E}}
\def\R{\mathbb{R}}

\def\Sig{\varSigma}

\def\veps{\varepsilon}

\title{A tail inequality for quadratic forms of subgaussian random vectors}
\author[1]{Daniel Hsu}
\author[1,2]{Sham M.~Kakade}
\author[3]{Tong Zhang}
\affil[1]{Microsoft Research New England}
\affil[2]{Department of Statistics, Wharton School, University of Pennsylvania}
\affil[3]{Department of Statistics, Rutgers University}

\begin{document}
\maketitle
{\def\thefootnote{}
\footnotetext{E-mail: \texttt{dahsu@microsoft.com},
\texttt{skakade@wharton.upenn.edu}, \texttt{tzhang@stat.rutgers.edu}}}

\begin{abstract}
We prove an exponential probability tail inequality for positive
semidefinite quadratic forms in a subgaussian random vector.
The bound is analogous to one that holds when the vector has independent
Gaussian entries.
\end{abstract}

\section{Introduction}

Suppose that $x=(x_1,\dotsc,x_n)$ is a random vector.  Let $A \in
\R^{m \times n}$ be a fixed matrix. A natural quantity that arises in many
settings is the quadratic form $\|Ax\|^2 = x^\top (A^\top A) x$.
Throughout $\|v\|$ denotes the Euclidean norm of a vector $v$, and
$\|M\|$ denotes the spectral (operator) norm of a matrix $M$.
We are interested in how close $\|Ax\|^2 $ is to its expectation.

Consider the special case where $x_1,\ldots,x_n$ are independent standard
Gaussian random variables.
The following proposition provides an (upper) tail bound for $\|Ax\|^2$.

\begin{proposition}
\label{proposition:quadratic-gaussian}
Let $A \in \R^{m \times n}$ be a matrix, and let $\Sig := A^\top A$.
Let $x = (x_1,\dotsc,x_n)$ be an isotropic multivariate Gaussian random
vector with mean zero.
For all $t > 0$,
\[
\Pr\left[ \|Ax\|^2 > \tr(\Sig) + 2\sqrt{\tr(\Sig^2) t} +
2\|\Sig\| t \right]
\leq e^{-t}
.
\]
\end{proposition}
The proof, given in Appendix~\ref{appendix:gaussian}, is straightforward given
the rotational invariance of the multivariate Gaussian distribution,
together with a tail bound for linear combinations of $\chi^2$ random
variables due to~\citet{LauMas00}.
We note that a slightly weaker form of
Proposition~\ref{proposition:quadratic-gaussian} can be proved directly
using Gaussian concentration~\citep{Pisier89}.

In this note, we consider the case where $x = (x_1,\dotsc,x_n)$ is a
\emph{subgaussian} random vector.
By this, we mean that there exists a $\sigma \geq 0$, such that for all
$\alpha \in \R^n$,
\begin{equation*}
\E\left[ \exp\left(\alpha^\top x\right) \right]
\leq \exp\left(\|\alpha\|^2 \sigma^2/2\right)
.
\end{equation*}
We provide a sharp upper tail bound for this case analogous to one that
holds in the Gaussian case (indeed, the same as
Proposition~\ref{proposition:quadratic-gaussian} when $\sigma = 1$).

\subsection*{Tail inequalities for sums of random vectors}
One motivation for our main result comes from the following observations
about sums of random vectors.
Let $a_1,\dotsc,a_n$ be vectors in a Euclidean space, and let $A = [
a_1|\dotsb|a_n ]$ be the matrix with $a_i$ as its $i$th column.
Consider the squared norm of the random sum
\begin{equation} \label{eq:sqnorm}
\|Ax\|^2 = \biggl\| \sum_{i=1}^n a_i x_i \biggr\|^2
\end{equation}
where $x := (x_1,\dotsc,x_n)$ is a martingale difference sequence with
$\E[x_i\ | \ x_1,\dotsc,x_{i-1}] = 0$ and $\E[x_i^2\ | \
x_1,\dotsc,x_{i-1}] = \sigma^2$.
Under mild boundedness assumptions on the $x_i$, the probability
that the squared norm in~\eqref{eq:sqnorm} is much larger than its
expectation
\[
\E[ \|Ax\|^2 ]
= \sigma^2 \sum_{i=1}^n \|a_i\|^2
= \sigma^2 \tr(A^\top A)
\]
falls off exponentially fast.
This can be shown, for instance, using the following lemma by taking $u_i =
a_ix_i$ (the proof is standard, but we give it for completeness in
Appendix~\ref{appendix:martingale}).
\begin{proposition}
\label{proposition:vector-bernstein}
Let $u_1,\dotsc,u_n$ be a martingale difference vector sequence
(\emph{i.e.}, $\E[u_i|u_1,\dotsc,u_{i-1}] = 0$ for all $i=1,\dotsc,n$) such
that
\[
\sum_{i=1}^n \E\bigl[ \|u_i\|^2 \ | \ u_1,\dotsc,u_{i-1} \bigl] \leq v
\quad \text{and} \quad
\|u_i\| \leq b
\]
for all $i=1,\dotsc,n$, almost surely.
For all $t > 0$,
\[
\Pr\left[
\biggl\| \sum_{i=1}^n u_i \biggr\|
> \sqrt{v} + \sqrt{8vt} + (4/3) bt
\right] \leq e^{-t}
.
\]
\end{proposition}
After squaring the quantities in the stated probabilistic event,
Proposition~\ref{proposition:vector-bernstein} gives the bound 
\[
\|Ax\|^2 \leq \sigma^2 \cdot \tr(A^\top A)
+ \sigma^2 \cdot O\left(
\tr(A^\top A) (\sqrt{t} + t)
+ \sqrt{\tr(A^\top A)} \max_i \|a_i\| (t + t^{3/2})
+ \max_i \|a_i\|^2 t^2 \right)
\]
with probability at least $1-e^{-t}$ when the $x_i$ are almost surely
bounded by $1$ (or any constant).

Unfortunately, this bound obtained from
Proposition~\ref{proposition:vector-bernstein} can be suboptimal when the
$x_i$ are subgaussian.
For instance, if the $x_i$ are Rademacher random variables, so $\Pr[x_i=+1]
= \Pr[x_i=-1] = 1/2$, then it is known that
\begin{equation} \label{eq:hw}
\|Ax\|^2 \leq \tr(A^\top A)
+ O\left( \sqrt{\tr((A^\top A)^2) t} + \|A\|^2 t \right)
\end{equation}
with probability at least $1-e^{-t}$.
A similar result holds for any subgaussian distribution on the
$x_i$~\citep{HanWri71}.
This is an improvement over the previous bound because the deviation terms
(\emph{i.e.}, those involving $t$) can be significantly smaller, especially
for large $t$.

In this work, we give a simple proof of~\eqref{eq:hw} with explicit
constants that match the analogous bound when the $x_i$ are independent
standard Gaussian random variables.

\section{Positive semidefinite quadratic forms}

Our main theorem, given below, is a generalization of~\eqref{eq:hw}.
\begin{theorem}
\label{theorem:quadratic-subgaussian}
Let $A \in \R^{m \times n}$ be a matrix, and let $\Sig := A^\top A$.
Suppose that $x = (x_1,\dotsc,x_n)$ is a random vector such that, for some
$\mu \in \R^n$ and $\sigma \geq 0$,
\begin{equation} \label{eq:subgaussian}
\E\left[ \exp\left(\alpha^\top(x-\mu)\right) \right]
\leq \exp\left(\|\alpha\|^2 \sigma^2/2\right)
\end{equation}
for all $\alpha \in \R^n$.
For all $t > 0$,
\begin{multline*}
\Pr\Biggl[\ \|Ax\|^2 > \sigma^2 \cdot \Bigl( \tr(\Sig)
+ 2\sqrt{\tr(\Sig^2) t}
+ 2\|\Sig\| t \Bigr)
+ \|A\mu\|^2 \cdot \biggl(1 + 4\biggl( \frac{\|\Sig\|^2}{\tr(\Sig^2)} t
\biggr)^{1/2} + \frac{4\|\Sig\|^2}{\tr(\Sig^2)} t
\biggr)^{1/2}
\ \Biggr] \leq e^{-t}
.
\end{multline*}
\end{theorem}
\begin{remark}
Note that when $\mu=0$ and $\sigma=1$ we have:
\[
\Pr\left[\ \|Ax\|^2 > \tr(\Sig)
+ 2\sqrt{\tr(\Sig^2) t}
+ 2\|\Sig\| t 
\ \right] \leq e^{-t}
\]
which is the same as Proposition~\ref{proposition:quadratic-gaussian}.
\end{remark}
\begin{remark}
Our proof actually establishes the following upper bounds on the moment
generating function of $\|Ax\|^2$ for $0 \leq \eta <
1/(2\sigma^2\|\Sig\|)$:
\begin{align*}
\E\left[ \exp\left( \eta \|Ax\|^2 \right) \right]
& \leq \E\left[ \exp\left( \sigma^2 \|A^\top z\|^2 \eta
+ \mu^\top A^\top z \sqrt{2\eta} \right) \right]
\\
& \leq \exp\left( \sigma^2\tr(\Sig) \eta + \frac{\sigma^4\tr(\Sig^2)\eta^2 +
\|A\mu\|^2 \eta}{1 - 2\sigma^2\|\Sig\|\eta} \right)
\end{align*}
where $z$ is a vector of $m$ independent standard Gaussian random
variables.
\end{remark}
\begin{proof}[Proof of Theorem~\ref{theorem:quadratic-subgaussian}]
Let $z$ be a vector of $m$ independent standard Gaussian random
variables (sampled independently of $x$).
For any $\alpha \in \R^m$,
\[ \E\left[ \exp\left( z^\top \alpha \right) \right]
= \exp\left( \|\alpha\|^2 / 2 \right)
.
\]
Thus, for any $\lambda \in \R$ and $\veps \geq 0$,
\begin{align}
\E\left[ \exp\left( \lambda z^\top Ax \right) \right]
& \geq \E\left[ \exp\left( \lambda z^\top Ax \right) \ \bigg| \ \|Ax\|^2
> \veps \right] \cdot \Pr\left[ \|Ax\|^2 > \veps \right]
\nonumber \\
& \geq \exp\left( \frac{\lambda^2\veps}2 \right) \cdot \Pr\left[ \|Ax\|^2
> \veps \right]
.
\label{eq:quadratic-form-lb}
\end{align}
Moreover,
\begin{align}
\E\left[ \exp\left( \lambda z^\top Ax \right) \right]
& = \E\left[ \E\left[ \exp\left( \lambda z^\top A(x-\mu) \right) \ \bigg| \
z \right] \exp\left( \lambda z^\top A\mu \right) \right] \nonumber \\
& \leq \E\left[ \exp\left( \frac{\lambda^2 \sigma^2}2 \|A^\top z\|^2 + 
\lambda \mu^\top A^\top z \right) \right]
\label{eq:quadratic-form-ub1}
\end{align}
Let $USV^\top$ be a singular value decomposition of $A$; where $U$ and $V$
are, respectively, matrices of orthonormal left and right singular vectors;
and $S = \diag(\sqrt{\rho_1},\dotsc,\sqrt{\rho_m})$ is the diagonal matrix
of corresponding singular values.
Note that
\[
\|\rho\|_1 = \sum_{i=1}^m \rho_i = \tr(\Sig)
,
\quad
\|\rho\|_2^2 = \sum_{i=1}^m \rho_i^2 = \tr(\Sig^2)
,
\quad \text{and} \quad
\|\rho\|_\infty = \max_i \rho_i = \|\Sig\|
.
\]
By rotational invariance, $y := U^\top z$ is an isotropic multivariate
Gaussian random vector with mean zero.
Therefore $\|A^\top z\|^2 = z^\top U S^2 U^\top z = \rho_1 y_1^2 +
\dotsb + \rho_m y_m^2$ and $\mu^\top A^\top z = \nu^\top y = \nu_1 y_1 +
\dotsb + \nu_m y_m$, where $\nu := SV^\top\mu$ (note that $\|\nu\|^2 =
\|SV^\top\mu\|^2 = \|A\mu\|^2$).
Let $\gamma := \lambda^2\sigma^2/2$.
By Lemma~\ref{lemma:gaussian-cgf},
\begin{equation} \label{eq:quadratic-form-ub2}
\E\left[ \exp\left( \gamma \sum_{i=1}^m \rho_i y_i^2 +
\frac{\sqrt{2\gamma}}{\sigma} \sum_{i=1}^m \nu_i y_i \right) \right]
\leq \exp\left( \|\rho\|_1 \gamma + \frac{\|\rho\|_2^2 \gamma^2 +
\|\nu\|^2 \gamma / \sigma^2}{1 - 2\|\rho\|_\infty \gamma} \right)
\end{equation}
for $0 \leq \gamma < 1/(2\|\rho\|_\infty)$.
Combining~\eqref{eq:quadratic-form-lb}, \eqref{eq:quadratic-form-ub1}, and
\eqref{eq:quadratic-form-ub2} gives
\begin{equation*}
\Pr\left[ \|Ax\|^2 > \veps \right]
\leq \exp\left( - \veps\gamma/\sigma^2 + \|\rho\|_1 \gamma +
\frac{\|\rho\|_2^2 \gamma^2 + \|\nu\|^2 \gamma / \sigma^2}{1 -
2\|\rho\|_\infty \gamma}
\right)
\end{equation*}
for $0 \leq \gamma < 1/(2\|\rho\|_\infty)$ and $\veps \geq 0$.
Choosing
\begin{align*}
\veps := \sigma^2 (\|\rho\|_1 + \tau) + \|\nu\|^2 \sqrt{1 +
\frac{2\|\rho\|_\infty \tau}{\|\rho\|_2^2}}
\quad \text{and} \quad
\gamma := \frac1{2\|\rho\|_\infty} \left(1 -
\sqrt{\frac{\|\rho\|_2^2}{\|\rho\|_2^2 + 2\|\rho\|_\infty \tau}}
\right)
,
\end{align*}
we have
\begin{align*}
\Pr\Biggl[\ \|Ax\|^2 > \sigma^2 (\|\rho\|_1 + \tau)
+ \|\nu\|^2 \sqrt{1 + \frac{2\|\rho\|_\infty \tau}{\|\rho\|_2^2}}
\ \Biggr]
& \leq \exp\Biggl( - \frac{\|\rho\|_2^2}{2\|\rho\|_\infty^2} \Biggl(\ 1 +
\frac{\|\rho\|_\infty \tau}{\|\rho\|_2^2} -\sqrt{1 +
\frac{2\|\rho\|_\infty \tau}{\|\rho\|_2^2}} \ \Biggr) \Biggr)
\\
& = \exp\Biggl( - \frac{\|\rho\|_2^2}{2\|\rho\|_\infty^2} h_1\Biggl(
\frac{\|\rho\|_\infty\tau}{\|\rho\|_2^2} \Biggr) \Biggr)
\end{align*}
where $h_1(a) := 1 + a - \sqrt{1+2a}$, which has the inverse function
$h_1^{-1}(b) = \sqrt{2b} + b$.
The result follows by setting $\tau := 2\sqrt{\|\rho\|_2^2t} +
2\|\rho\|_\infty t = 2\sqrt{\tr(\Sig^2)t} + 2\|\Sig\| t$.
\end{proof}

The following lemma is a standard estimate of the logarithmic moment
generating function of a quadratic form in standard Gaussian random
variables, proved much along the lines of the estimate due
to~\citet{LauMas00}.
\begin{lemma} \label{lemma:gaussian-cgf}
Let $z$ be a vector of $m$ independent standard Gaussian random variables.
Fix any non-negative vector $\alpha \in \R_+^m$ and any vector $\beta \in
\R^m$.
If $0 \leq \lambda < 1/(2\|\alpha\|_\infty)$, then
\[
\log \E\left[ \exp\left( \lambda \sum_{i=1}^m \alpha_i z_i^2 + \sum_{i=1}^m
\beta_i z_i \right) \right]
\leq
\|\alpha\|_1 \lambda + \frac{\|\alpha\|_2^2 \lambda^2 +
\|\beta\|_2^2/2}{1 - 2\|\alpha\|_\infty \lambda}
.
\]
\end{lemma}
\begin{proof}
Fix $\lambda \in \R$ such that $0 \leq \lambda < 1/(2\|\alpha\|_\infty)$,
and let $\eta_i := 1 / \sqrt{1 - 2\alpha_i\lambda} > 0$ for $i=1,\dotsc,m$.
We have
\begin{align*}
\E\left[ \exp\left( \lambda \alpha_i z_i^2 + \beta_i z_i \right) \right]
& = \int_{-\infty}^\infty
\frac1{\sqrt{2\pi}} \exp\left( - z_i^2/2 \right) \exp\left(
\lambda \alpha_i z_i^2 + \beta_i z_i \right) dz_i \\
& = \eta_i
\exp\left( \frac{\beta_i^2\eta_i^2}2 \right)
\int_{-\infty}^\infty
\frac1{\sqrt{2\pi\eta_i^2}}
\exp\left( - \frac1{2\eta_i^2} \left(z_i - \beta_i\eta_i^2 \right)^2
\right)
dz_i
\end{align*}
so
\begin{align*}
\log \E\left[ \exp\left( \lambda \sum_{i=1}^m \alpha_i z_i^2 + \sum_{i=1}^m
\beta_i z_i \right) \right]
= \frac12 \sum_{i=1}^m \beta_i^2\eta_i^2
+ \frac12 \sum_{i=1}^m \log \eta_i^2
.
\end{align*}
The right-hand side can be bounded using the inequalities
\begin{equation*}
\frac12 \sum_{i=1}^m \log \eta_i^2
= -\frac12 \sum_{i=1}^m \log(1 - 2\alpha_i \lambda)
= \frac12 \sum_{i=1}^m \sum_{j=1}^\infty \frac{(2\alpha_i\lambda)^j}{j}
\leq \|\alpha\|_1 \lambda + \frac{\|\alpha\|_2^2\lambda^2}{1 -
2\|\alpha\|_\infty\lambda}
\end{equation*}
and
\begin{equation*}
\frac12 \sum_{i=1}^m \beta_i^2\eta_i^2
\leq \frac{\|\beta\|_2^2/2}{1 - 2\|\alpha\|_\infty\lambda}
.
\qedhere
\end{equation*}
\end{proof}

\subsubsection*{Example: fixed-design regression with subgaussian noise}
We give a simple application of Theorem~\ref{theorem:quadratic-subgaussian}
to fixed-design linear regression with the ordinary least squares
estimator.

Let $x_1,\dotsc,x_n$ be fixed design vectors in $\R^d$.
Let the responses $y_1,\dotsc,y_n$ be random variables for which there
exists $\sigma > 0$ such that
\[
\E\left[ \exp\left( \sum_{i=1}^n \alpha_i (y_i - \E[y_i]) \right)
\right]
\leq \exp\left( \sigma^2 \sum_{i=1}^n \alpha_i^2 \right)
\]
for any $\alpha_1,\dotsc,\alpha_n \in \R$.
This condition is satisfied, for instance, if
\[ y_i = \E[y_i] + \veps_i \]
for independent subgaussian zero-mean noise variables
$\veps_1,\dotsc,\veps_n$.
Let $\Sig := \sum_{i=1}^n x_ix_i^\top / n$, which we assume is invertible
without loss of generality.
Let
\[
\beta := \Sig^{-1} \left( \frac1n \sum_{i=1}^n x_i\E[y_i] \right)
\]
be the coefficient vector of minimum expected squared error.
The ordinary least squares estimator is given by
\[
\hat\beta := \Sig^{-1} \left( \frac1n \sum_{i=1}^n x_iy_i \right)
.
\]
The excess loss $R(\hat\beta)$ of $\hat\beta$ is the difference between the
expected squared error of $\hat\beta$ and that of $\beta$:
\[
R(\hat\beta)
:= \E\left[ \frac1n \sum_{i=1}^n (x_i^\top\hat\beta-y_i)^2 \right]
- \E\left[ \frac1n \sum_{i=1}^n (x_i^\top\beta-y_i)^2 \right]
.
\]
It is easy to see that
\[
R(\hat\beta)
= \bigl\|\Sig^{1/2} (\hat\beta - \beta) \bigr\|^2
= \Bigl\|\sum_{i=1}^n \bigl(\Sig^{-1/2} x_i\bigr)(y_i-\E[y_i])\Bigr\|^2
.
\]
By Theorem~\ref{theorem:quadratic-subgaussian},
\[
\Pr\left[
R(\hat\beta)
> \frac{\sigma^2 \bigl( d + 2\sqrt{dt} + 2t \bigr)}{n}
\right] \leq e^{-t}
.
\]
Note that in the case that $\E[(y_i - \E[y_i])^2] = \sigma^2$ for
each $i$, then
\[
\E[R(\hat\beta)] = \frac{\sigma^2 d}{n}
;
\]
so the tail inequality above is essentially tight when the $y_i$ are
independent Gaussian random variables.

\bibliography{../mle_linear}
\bibliographystyle{plainnat}

\appendix

\section{Standard tail inequalities}

\subsection{Martingale tail inequalities}
\label{appendix:martingale}

The following is a standard form of Bernstein's inequality stated for
martingale difference sequences.
\begin{lemma}[Bernstein's inequality for martingales]
\label{lemma:bernstein}
Let $d_1,\dotsc,d_n$ be a martingale difference sequence with respect to
random variables $x_1,\dotsc,x_n$ (\emph{i.e.}, $\E[d_i|x_1,\dotsc,x_{i-1}]
= 0$ for all $i = 1,\dotsc,n$) such that $|d_i| \leq b$ and $\sum_{i=1}^n
\E[d_i^2|x_1,\dotsc,x_{i-1}] \leq v$.
For all $t > 0$,
\[
\Pr\left[
\sum_{i=1}^n d_i > \sqrt{2vt} + (2/3)bt
\right] \leq e^{-t}
.
\]
\end{lemma}

The proof of Proposition~\ref{proposition:vector-bernstein}, which is
entirely standard, is an immediate consequence of the following two lemmas
together with Jensen's inequality.
\begin{lemma}
\label{lemma:bernstein-l2}
Let $u_1,\dotsc,u_n$ be random vectors such that
\[
\sum_{i=1}^n \E\left[ \|u_i\|^2 \ | \ u_1,\dotsc,u_{i-1} \right] \leq v
\quad \text{and} \quad
\|u_i\| \leq b
.
\]
for all $i=1,\dotsc,n$, almost surely.
For all $t > 0$,
\[
\Pr\left[\
\biggl\| \sum_{i=1}^n u_i \biggr\|
- \E\biggl[ \biggl\| \sum_{i=1}^n u_i \biggr\| \biggr]
> \sqrt{8vt}
+ (4/3) b t
\ \right] \leq e^{-t}
.
\]
\end{lemma}
\begin{proof}
Let $s_n := u_1 + \dotsb + u_n$.
Define the Doob martingale
\[ d_i := \E[ \|s_n\| \ | \ u_1,\dotsc,u_i ] - \E[ \|s_n\| \ | \
u_1,\dotsc,u_{i-1} ] \]
for $i = 1,\dotsc,n$, so $d_1 + \dotsb + d_n = \|s_n\| - \E[\|s_n\|]$.
First, clearly, $\E[d_i | u_1,\dotsc,u_{i-1}] = 0$.
Next, the triangle inequality implies
\begin{align*}
d_i
& = \E\left[ \|(s_n - u_i) + u_i\| \ | \ u_1,\dotsc,u_i\right]
- \E\left[ \|(s_n - u_i) + u_i\| \ | \ u_1,\dotsc,u_{i-1}\right] \\
& \leq \E\left[ \|s_n - u_i\| + \|u_i\| \ | \ u_1,\dotsc,u_i\right]
- \E\left[ \| \|s_n - u_i\| - \|u_i\| \ | \ u_1,\dotsc,u_{i-1}\right] \\
& = \|u_i\| + \E\left[ \|u_i\| \ | \ u_1,\dotsc,u_{i-1}\right]
,
\\
\text{and similarly},
\
d_i
& \geq -\|u_i\| - \E\left[ \|u_i\| \ | \ u_1,\dotsc,u_{i-1}\right]
.
\end{align*}
Therefore,
\[
|d_i|
\leq \|u_i\| + \E\left[ \|u_i\| \ | \ u_1,\dotsc,u_{i-1}\right] \\
\leq 2b
\quad \text{almost surely}
.
\]
Moreover,
\begin{align*}
\E\left[ d_i^2 \ | \ u_1,\dotsc,u_{i-1} \right]
& \leq \E\Bigl[ \|u_i\|^2
+ 2 \cdot \|u_i\| \cdot \E\left[ \|u_i\| \ | \ u_1,\dotsc,u_{i-1} \right]
\\
& \quad\qquad{}
+ \E\left[ \|u_i\| \ | \ u_1,\dotsc,u_{i-1} \right]^2 \ | \
u_1,\dotsc,u_{i-1} \Bigr] \\
& = \E\left[ \|u_i\|^2 \ | \ u_1,\dotsc,u_{i-1} \right]
+ 3 \cdot \E\left[ \|u_i\| \ | \ u_1,\dotsc,u_{i-1} \right]^2 \\
& \leq 4 \cdot \E\left[ \|u_i\|^2 \ | \ u_1,\dotsc,u_{i-1} \right]
,
\\
\text{so} \quad
\sum_{i=1}^n \E\left[ d_i^2 \ | \ u_1,\dotsc,u_{i-1} \right]
& \leq 4v \quad \text{almost surely}
.
\end{align*}
The claim now follows from Bernstein's inequality
(Lemma~\ref{lemma:bernstein}).
\end{proof}

\begin{lemma}
\label{lemma:expected-length}
If $u_1,\dotsc,u_n$ is a martingale difference vector sequence
(\emph{i.e.}, $\E[u_i|u_1,\dotsc,u_{i-1}] = 0$ for all $i=1,\dotsc,n$), then
\[
\E\biggl[ \biggl\| \sum_{i=1}^n u_i \biggr\|^2 \biggr]
= \sum_{i=1}^n \E\left[ \|u_i\|^2 \right]
.
\]
\end{lemma}
\begin{proof}
Let $s_i := u_1 + \dotsb + u_i$ for $i = 1,\dotsc,n$; we have
\begin{align*}
\E\left[ \|s_n\|^2 \right]
& = \E\left[ \E\left[ \|u_n + s_{n-1}\|^2 \ | \ u_1,\dotsc,u_{n-1}
\right] \right] \\
& = \E\left[ \E\left[ \|u_n\|^2 + 2u_n^\top s_{n-1} + \|s_{n-1}\|^2 \ |
\ u_1,\dotsc,u_{n-1} \right] \right] \\
& = \E\left[ \|u_n\|^2 \right] + \E\left[ \|s_{n-1}\|^2 \right]
\end{align*}
so the claim follows by induction.
\end{proof}

\subsection{Gaussian quadratic forms and $\chi^2$ tail inequalities}
\label{appendix:gaussian}

It is well-known that if $z \sim \mathcal{N}(0,1)$ is a standard Gaussian
random variable, then $z^2$ follows a $\chi^2$ distribution with one degree
of freedom.
The following inequality due to~\citet{LauMas00} gives a bound on linear
combinations of $\chi^2$ random variables.
\begin{lemma}[$\chi^2$ tail inequality; \citealp{LauMas00}]
\label{lemma:chi2-tail}
Let $q_1,\dotsc,q_n$ be independent $\chi^2$ random variables, each
with one degree of freedom.
For any vector $\gamma = (\gamma_1,\dotsc,\gamma_n) \in \R_+^n$ with
non-negative entries, and any $t > 0$,
\[
\Pr\left[ \sum_{i=1}^n \gamma_i q_i > \|\gamma\|_1 +
2\sqrt{\|\gamma\|_2^2 t} + 2\|\gamma\|_\infty t
\right] \leq e^{-t}
.
\]
\end{lemma}

\begin{proof}[Proof of Proposition~\ref{proposition:quadratic-gaussian}]
Let $V \Lambda V^\top$ be an eigen-decomposition of $A^\top A$, where $V$
is a matrix of orthonormal eigenvectors, and $\Lambda :=
\diag(\rho_1,\dotsc,\rho_n)$ is the diagonal matrix of corresponding
eigenvalues $\rho_1,\dotsc,\rho_n$.
By the rotational invariance of the distribution, $z := V^\top x$ is an
isotropic multivariate Gaussian random vector with mean zero.
Thus, $\|Ax\|^2 = z^\top \Lambda z = \rho_1 z_1^2 + \dotsb + \rho_n
z_n^2$, and the $z_i^2$ are independent $\chi^2$ random variables, each
with one degree of freedom.
The claim now follows from a tail bound for $\chi^2$ random variables
(Lemma~\ref{lemma:chi2-tail}, due to~\citealp{LauMas00}).
\end{proof}
\end{document}